\documentclass[11pt,a4paper]{amsart}
\usepackage[utf8]{inputenc}
\usepackage[english]{babel}
\usepackage{pst-grad} % For gradients
\usepackage{pst-plot} % For axes
\usepackage{pstricks}
\usepackage{amsmath}
\usepackage{amssymb}
\usepackage{mathrsfs}
\usepackage{amsthm}
\usepackage[dvips]{graphicx}
\usepackage{epsfig}
\newcommand{\RR}{\mathbb R}

\newcommand{\ZZ}{\mathbb Z}

\newcommand{\TT}{\mathbb T}

\newcommand{\pat}{\partial_t}
\newcommand{\pax}{\partial_x}

\newcommand{\jeps}{\mathcal{H}_\epsilon*}

\newcommand{\vertiii}[1]{{\left\vert\kern-0.25ex\left\vert\kern-0.25ex\left\vert #1 
    \right\vert\kern-0.25ex\right\vert\kern-0.25ex\right\vert}}

\usepackage[pdfauthor={...},pdftitle={...},bookmarks,colorlinks]{hyperref}

\newcounter{comentcount}
\newcounter{teocount}
\newtheorem{lem}{Lemma}
\newtheorem{prop}{Proposition}

\newtheorem{teo}[teocount]{Theorem}  
\newtheorem{defi}{Definition}

\newtheorem{remark}{Remark}

\title[]{On the fractional Fisher information with applications to a hyperbolic-parabolic system of chemotaxis}
\author[R. Granero-Belinch\'{o}n]{Rafael Granero-Belinch\'{o}n}
\email{granero@math.univ-lyon1.fr}
\address{Univ Lyon, Université Claude Bernard Lyon 1, CNRS UMR 5208, Institut Camille Jordan, 43 blvd. du 11 novembre 1918, F-69622 Villeurbanne cedex, France.}

\begin{document}

\begin{abstract}
We introduce new lower bounds for the fractional Fisher information. Equipped with these bounds we study a hyperbolic-parabolic model of chemotaxis and prove the global existence of solutions in certain dissipation regimes.
\end{abstract}

\maketitle

%{\small
%\tableofcontents}

\section{Introduction}
In this note we study the following system of partial differential equations
\begin{equation}\label{eq:6b}
\left\{\begin{aligned}
\pat u&=-\mu\Lambda^\alpha u+\pax (uq),\\
\pat q&=\pax f(u),
\end{aligned}\right.\text{ for }x\in\TT,\,t\geq0,
\end{equation}
where $\TT$ denotes the $1-$dimensional torus, $f$ is a smooth function, $\Lambda^\alpha=(-\Delta)^{\alpha/2}$ denotes the fractional Laplacian with $0<\alpha\leq 2$ (see Appendix \ref{sec9b} for the expression as a singular integral and some properties) and $\mu\geq0$ is a fixed constant. 

This system was proposed by Othmers \& Stevens \cite{stevens1997aggregation} (see also Levine, Sleeman, Brian \& Nilsen-Hamilton \cite{levine2000mathematical}) based on biological considerations as a model of the formation of new blood vessels from pre-existing blood vessels (in a process that is called tumor angiogenesis). In particular, in the previous system, $u$ is the density of vascular endothelial cells and $q=\pax\log(v)$ where $v$ is the concentration of the signal protein known as vascular endothelial growth factor (VEGF). As $f$ comes from the chemical kinetics of the system, it is commonly referred as the \emph{kinetic function}. The interested reader can refer to Bellomo, Li, \& Maini \cite{bellomo2008foundations} for a detailed exposition on tumor modelling. In the case where $f(u)=u^2/2$, equation \eqref{eq:6b} also appears as a viscous regularization of the dispersionless Majda-Biello model of the interaction of barotropic and equatorial baroclinic Rossby waves \cite{majda2003nonlinear}. Another related model is the magnetohydrodynamic-Burgers system proposed by Fleischer \& Diamond \cite{fleischer2000burgers} (see also Jin, Wang \& Xiong \cite{jin2015cauchy} and the references therein).

We address the existence of solutions and their qualitative properties in the case $0<\alpha<2$. In particular, among other results, we prove the global existence of weak solutions for $f(u)=u^r/r$, $1\leq r\leq 2$ and $\alpha>2-r$. This topic is mathematically challenging due to the hyperbolic character of the equation for $q$. Indeed, at least formally, the velocity $q$ is one derivative less regular than $u$. So, the term $\pax (u q)$ is two derivatives less regular than $u$. This suggests that the diffusion given by the Laplacian ($\alpha=2$) is somehow critical. 

The main tool to achieve the results is a set of new inequalities for the generalized Fisher information (see \cite{toscani2015fractional} for a similar functional)
\begin{equation}\label{fisher}
\mathcal{I}_\alpha=\int_\TT (-\Delta)^{\alpha/2}u\Gamma(u)dx,
\end{equation}
where $\Gamma$ is a smooth increasing function. This functional is a generalization of the classical Fisher information (also known as Linnik functional)
\begin{equation}\label{fisher2}
\mathcal{I}_2=\int_\TT -\Delta u\log(u)dx,
\end{equation}
introduced in Fisher \cite{fisher1925theory} (see also Linnik \cite{linnik1959information}, McKean \cite{mckean1966speed}, Toscani \cite{toscani1992lyapunov, toscani1992new}, Villani \cite{villani1998fisher}). The Fisher information appears commonly as the rate at which the Shannon's entropy\footnote{To be completely precise, the original Shannon's entropy is $-\mathcal{S}$ and not $\mathcal{S}$ itself.} \cite{shannon1948mathematical} (or, equivalently, the Boltzmann's H function)
\begin{equation}\label{shannon}
\mathcal{S}=\int_\TT u\log(u)dx,
\end{equation}
is dissipated by diffusive semigroups as, for instance, the semigroup generated by the linear heat equation.

\subsection*{Motivation} Our motivation is two-fold. On the one hand, our motivation comes from the mathematical modelling of cancer angiogenesis. On the other hand, we find the new functional inequalities involving the fractional Fisher information that are interesting by themselves.

\subsubsection*{Mathematical biology}
The \emph{classical} Keller-Segel model of chemotaxis reads
\begin{equation}\label{eq:1}
\left\{\begin{aligned}
\pat u&=\Delta u-\chi\nabla\cdot (u\nabla\Phi(v))+F(u),\\
\tau\pat v&=\nu\Delta v+G(u,v),
\end{aligned}\right.
\end{equation}
for given functions $\Phi, F,G$. Here $u$ is the cell density and $v$ denotes again the chemical concentration. The sign of $\chi$ in \eqref{eq:1} plays an important role: $\text{sgn}(\chi)$ indicates whether we have attraction effects ($\text{sgn}(\chi)=1$) or repulsive effects ($\text{sgn}(\chi)=-1$). This model was originally proposed by Keller \& Segel \cite{keller1970initiation} (see also Patlak \cite{patlak1953random}) as a model of aggregation of the slime mold \emph{Dictyostelium discoideum}. There is a huge literature on the mathematical study of the numerous versions of \eqref{eq:1}. The interested reader can refer to the works by Blanchet, Carlen \& Carrillo \cite{blanchet2010functional},  Blanchet, Carrillo \& Masmoudi \cite{BCM}, Calvez \& Carrillo \cite{Carrillo2}, Dolbeault \& Perthame \cite{Dolbeault2} and the references therein. The applications of the system \eqref{eq:1} are wide. For instance, the model 
\begin{equation}\label{eq:2}
\left\{\begin{aligned}
\pat u&=\Delta u-\chi\nabla\cdot (u\nabla v)+ru(1-u),\\
\tau\pat v&=\nu\Delta v-v+u,
\end{aligned}\right.
\end{equation}
is related with the three-component urokinase plasminogen invasion mode (see the works Hillen, Painter \& Winkler \cite{Hillen1}). Let us remark that equation \eqref{eq:2} is equivalent to \eqref{eq:1} if $\Phi(v)=v$, $F(u)=ru(1-u)$ and $G(u,v)=u-v$. For a mathematical analysis of the previous model, the interested reader can check the works by, Tello \& Winkler \cite{TelloWinkler}, Burczak \& Granero-Belinch\'on \cite{BG, BG2, BG4} and the references therein.

Another example arises when we take into consideration the fact that cancer cells can also move according to the gradients of the stiff tissue (haptotaxis), we arrive to the coupled chemotaxis-haptotaxis model
\begin{equation}\label{eq:3}
\left\{\begin{aligned}\pat u& = \Delta u-\chi\nabla\cdot (u\nabla\Phi(v))-\xi\nabla\cdot (u\nabla\Psi(w))+F(u),\\
\tau\pat v&=\nu\Delta v+G(u,v),\\
\pat w&=H(u,v,w).
\end{aligned}\right..
\end{equation}
Here $u,v,w$ represents cell density, enzyme concentration and tissue density, respectively. 

For a more complete discussion on these models, the interested reader can check the extensive surveys by Hillen \& Painter \cite{Hillen3} and Bellomo, Bellouquid, Tao \& Winkler \cite{bellomo2015towards} and Blanchet \cite{blanchet2011parabolic}. 

Also, it has been suggested by experiments and observation that the feeding strategies of some species should be modelled using L\'evy processes. For instance one can refer to the works by Raichlen, Wood, Gordon, Mabulla, Marlowe \& Pontzer \cite{hadza} (see also the introduction in \cite[Section 2]{BG4}). Thus, a model with a fractional laplacian instead of a laplacian arises from applications:

\begin{equation}\label{eq:4}
\left\{\begin{aligned}
\pat u&=-(-\Delta)^{\alpha/2} u-\chi\nabla\cdot (u\nabla\Phi(v))+F(u),\\
\tau\pat v&=-\nu(-\Delta)^{\beta/2} v+G(u,v),
\end{aligned}\right..
\end{equation}

Once \eqref{eq:4} is considered, there are two limiting cases that are of particular importance. The first one is when the diffusion of the chemical, $v$, is much faster than the movement of the cells. Mathematically, this implies the limit $\tau\rightarrow0$ and corresponds to the, so called, parabolic-elliptic Keller-Segel system
\begin{equation}\label{eq:5}
\left\{\begin{aligned}
\pat u&=-(-\Delta)^{\alpha/2} u-\chi\nabla\cdot (u\nabla\Phi(v))+F(u),\\
0&=-(-\Delta)^{\beta/2} v+G(u,v),
\end{aligned}\right..
\end{equation}

Equation \eqref{eq:5} also appears related with the formation of large-scale structure in the primordial universe (see Ascasibar, Granero-Belinch\'on \& Moreno \cite{AGM}) or semiconductor devices (Granero-Belinch\'on \cite{Gsemiconductor}). 

The interested reader in the mathematical study on models akin to \eqref{eq:4} and \eqref{eq:5} can read the papers by Escudero \cite{escudero2006fractional}, Li, Rodrigo \& Zhang \cite{li2010exploding}, Bournaveas \& Calvez \cite{bournaveas2010one}, Burczak \& Granero \cite{BG3, BG}, Granero-Belinch\'on \& Orive \cite{GO}, Biler \& Wu \cite{BilerWu} and Wu \& Zheng \cite{WuZheng}.

The second limit case arises when the diffusion of the chemical is negligible. In that case we have $\nu\rightarrow0$ and we recover the following hybrid PDE-ODE system
\begin{equation}\label{eq:6}
\left\{\begin{aligned}
\pat u&=-\mu(-\Delta)^{\alpha/2} u-\chi\nabla\cdot (u\nabla\Phi(v))+F(u),\\
\tau\pat v&=G(u,v),
\end{aligned}\right..
\end{equation}

Notice that for the particular choice of $\tau=1$, $\chi=-1$, $F=0$, $\Phi(v)=\log(v)$, $G(u,v)=(f(u)+\lambda)v$ for certain smooth, non-decreasing function $f:\RR^+\rightarrow\RR$ and $\lambda\in\RR$, we have that \eqref{eq:6} is equivalent to \eqref{eq:6b} (recall that the new variable is $q=\nabla\log(v)$).  

\subsubsection*{Functional inequalities} For many parabolic equations, the Shannon's entropy \eqref{shannon} (or close variants) is dissipated with a rate proportional to the Fisher information \eqref{fisher}. For instance, for the fractional heat equation one has the following equality
$$
\frac{d}{dt}\mathcal{S}(t)=-\mathcal{I}_\alpha.
$$
The Shannon's entropy plays an important role in the Patlak-Keller-Segel equation \eqref{eq:5} (see Dolbeault \& Perthame \cite{Dolbeault2} and Blanchet, Dolbeault \& Perthame \cite{Dolbeault3}). Actually, for \eqref{eq:5} with $\beta=2$, $G=u-\langle u \rangle$, $\Phi=v$ and $F=0$, we have the equality
$$
\frac{d}{dt}\mathcal{S}(t)=-\mathcal{I}_\alpha+\int (u-\langle u \rangle)u.
$$
Examples of other equations with similar entropy-entropy production equalities are 

\begin{itemize}
\item nonlinear reaction-diffusion systems modeling reversible chemical reactions (see for instance Mielke, Haskovec \& Markowich \cite{uniformdecayentropy}) as
\begin{align*}
\pat u&= \Delta u + v^2-u\\
\pat v&= \Delta v - (v^2-u).
\end{align*}
\item a one dimensional model of the two-dimensional Surface Quasi-Geostrophic equation involving the Hilbert transform $H$
$$
\pat u=-\pax(u Hu)
$$
(see Castro \& C\'ordoba \cite{CC}, Carrillo, Ferreira \& Precioso \cite{Carrillo} Cafarelli \& V\'azquez \cite{CV}, Bae \& Granero-Belinch\'on \cite{bae2015global}).
\item a model for the slope of the interface in two-phase flow in porous media
$$
\pat u=-\pax\left (\frac{Hu}{1+u^2}\right),
$$
(see Granero-Belinch\'on, Navarro, \& Ortega \cite{GNO}). 
\end{itemize}
Consequently, lower bounds for the Fisher information allow to obtain parabolic gain of regularity of the type $L^p_t W^{s,q}_x$ and, together with Poincar\'e or Sobolev inequalities can be used to obtain explicit rates of convergence to equilibrium.

\subsection*{Prior results on the parabolic-hyperbolic system}
To the best of our knowledge, the only results on \eqref{eq:6b} that are available in the literature study the case where the diffusion is local ($\alpha=2$). 

Among these, the one dimensional case has received lots of attention in the recent years. In particular, Fan \& Zhao \cite{fan2012blow}, Li \& Zhao \cite{li2015initial}, Mei, Peng \& Wang \cite{mei2015asymptotic}, Li, Pan \& Zhao \cite{li2015quantitative}, Jun, Jixiong, Huijiang \& Changjiang \cite{jun2009global} Li \& Wang \cite{li2009nonlinear} and Zhang \& Zhu \cite{zhang2007global} studied the system \eqref{eq:6b} when $\alpha=2$ and $f(u)=u$ under different boundary conditions. Notably, they proved the global existence of classical solution and the asymptotic behavior for large times. Jin, Li \& Wang \cite{jin2013asymptotic} and Li, Li \& Wang \cite{li2014stability} studied the existence and stability of traveling waves. Wang \& Hillen studied the existence of shock solutions \cite{wang2008shock}.

The one dimensional case with $\alpha=2$ and a general $f(u)$ was studied by Zhang, Tan \& Sun \cite{zhang2013global} and Li \& Wang \cite{li2010nonlinear}. 

The proofs in the one dimensional case take advantage of the dissipative character of the system, namely, that
$$
-\int_0^t\int_\Omega \log(u)\pax^2u dxds<\infty.
$$
This dissipation is enough to guarantee global bounds $u\in L^p_tL^q_x$ that, conversely, implies a global bound $\pat q\in L^2_tL^2_x$. In the fractional case, the dissipation is weaker and the analogous bound for $\pat q$ is
$$
\pat q\in L^2_t H^{\alpha/2-1}_x.
$$

In high dimensions the results are more sparse. In that regard, Wang, Xiang \& Yu \cite{wang2016asymptotic} study the global well-posedness for small initial data of a viscosity regularization of a high-dimensional version of \eqref{eq:6b}, Li, Li \& Zhao \cite{li2011hyperbolic} studied the local existence, blow-up criteria and global existence of small data for the system \eqref{eq:6b} in 2 and 3 dimensions. Furthermore, they also proved the decay of certain Sobolev norms for small data. A global existence result regarding the multidimensional \emph{alter ego} of \eqref{eq:6b} in Besov spaces can be found in Hau \cite{hao2012global}. The global existence of solution for small initial data was address also by Li, Pan \& Zhao \cite{li2012global}.

In the forthcoming paper \cite{Ghyperparstrong}, we address the well-posedness of \eqref{eq:6b} in two spatial dimension when $f(u)=u^2/2$.

\subsection*{Plan of the paper} 
The plan of the paper is as follows. In section \ref{sec2} we state our results. In section \ref{sec3} we present some of the notation and the functional spaces. In section \ref{sec4} we prove our results for the fractional Fisher information. In section \ref{sec5} we prove the local existence of smooth solutions for \eqref{eq:6b}, while in section \ref{sec6} we prove the dissipative character of the system. In section \ref{sec7} we prove the global existence of solution when $\alpha=2$. In section \ref{sec8} we establish the global existence of weak solution for \eqref{eq:6b}. In Appendix \ref{sec9b} we obtain the explicit expression of the fractional Laplacian as a singular integral and some properties. Finally, in Appendix \ref{sec9} we write some auxiliary inequalities regarding fractional Sobolev spaces.

\section{Results and discussion}\label{sec2}
\subsection*{Results regarding the fractional Fisher information}
In this section we establish some lower bounds for the fractional Fisher information inequalities. These inequalities are generalizations of those in Bae \& Granero-Belinch\'on \cite{bae2015global} and Burczak, Granero-Belinch\'on \& Luli \cite{BGK}, Li \& Zhao \cite[equations 1.7, 2.20]{li2015initial} and Li, Pan \& Zhao \cite[equation 1.12]{li2015quantitative}. For the sake of generality, we consider the case where the spatial domain is either $\Omega^d=\RR^d$ or $\Omega^d=\TT^d$.

In what follows we assume that 
$$
\Gamma(z):\RR^+\rightarrow\RR
$$ 
is a fixed, $C^1$, increasing function such that
\begin{equation}\label{auxphi}
\Gamma'(z)\geq \frac{c}{z}\geq0,
\end{equation}
where $c$ is a fixed constant. For instance, an example of such a function $\Gamma$ would be $\Gamma(z)=\log(z)$ for $z>0$.
\begin{lem}\label{lemaentropy}
Let $d\geq1$, $0\leq u$ be a smooth, given function and $0<\alpha<2$, $0<\delta<\alpha/2$ be two fixed constants. Then, 
\begin{equation}\label{aux1}
\|u\|_{\dot{H}^{\alpha/2}(\Omega^d)}^2\leq C(\alpha,d,\Gamma)\|u\|_{L^\infty(\Omega^d)}\int_{\Omega^d}\Lambda^\alpha u(x)\Gamma(u(x))dx,\;\Omega^d=\RR^d,\TT^d,
\end{equation}
\begin{equation}\label{aux2}\|u\|_{\dot{W}^{\alpha/2-\delta,1}(\TT^d)}^2\leq C(\alpha,d,\delta, \Gamma)\|u\|_{L^1(\TT^d)}\int_{\TT^d}\Lambda^\alpha u(x)\Gamma(u(x))dx.
\end{equation}
\end{lem}

\begin{remark}
In the local case $\alpha=2$, \eqref{aux2} holds with $\delta=0$.
\end{remark}

In the one-dimensional case we can obtain a sharper result:

\begin{lem}\label{lemaentropy2}
Let $d=1$and $1<\alpha\leq2$ be a fixed constant. Given $0\leq u$ a smooth function, we have that 
\begin{equation}\label{eq:10c}
\|u\|_{\dot{H}^{\alpha/2}(\RR)}^{2-\frac{2}{1+\alpha}}\leq C(\alpha,\Gamma)\|u\|_{L^1(\RR)}^{1-\frac{2}{1+\alpha}}\int_{\RR}\Lambda^\alpha u(x)\Gamma(u(x))dx,
\end{equation}
\begin{equation}\label{eq:10d}
\|u\|_{\dot{H}^{\alpha/2}(\TT)}^{2-\frac{2}{1+\alpha}}\leq C(\alpha,\Gamma)\|u\|_{L^1(\TT)}^{1-\frac{2}{1+\alpha}}\left(\int_{\TT}\Lambda^\alpha u(x)\Gamma(u(x))dx+\|u\|_{L^1(\TT)}\right).
\end{equation}
\end{lem}

\subsection*{Results regarding equation (\ref{eq:6b})}
We start this section with the definition of \emph{admissible} kinetic function $f$:
\begin{defi}\label{admissible} A kinetic function 
$$
f(y):(-1,\infty)\rightarrow \RR^+
$$
is admissible if
\begin{itemize}
\item $f(y)\in W^{4,\infty}(-1,\infty)$,
\item $f'(y)>0$ if $y>0$,
\item for $y\in[a,b]\subset [0,\infty)$, there exist $\gamma_{a}^{b},\tilde{\gamma}_{a}^{b}<\infty$ such that
\begin{equation}\label{gamma}
0\leq \gamma_{a}^{b}\leq \frac{y}{f'(y)}\leq \tilde{\gamma}_a^b,\;\forall\, y\in[a,b].
\end{equation}
\end{itemize} 
\end{defi}

Sometimes we require the admissible kinetic function $f$ (see Definition \ref{admissible}) to satisfy the uniform bound 
\begin{equation}\label{lowerf'}
C_1\leq f'(y)\;\,\forall\,y\geq0,
\end{equation}
for suitable constant $C_1$.

Our first result establishes the local existence for the one-dimensional problem \eqref{eq:6b} for $\mu\geq0$, $0\leq \alpha\leq 2$ and a general admissible kinetic function. Let us emphasize that the proof of this result does not use the regularizing effect from the viscosity. In other words, the result holds true in the inviscid case $\mu=0$.
 
\begin{teo}\label{local} Let $(u_0,q_0)\in H^3(\TT)\times H^3(\TT)$ be the initial data and $f$ be an admissible kinetic function in terms of Definition \ref{admissible}. Assume that $u_0\geq0$, $\langle q_0\rangle=0$, $0\leq\alpha\leq2$, $\mu\geq0$. Then there exist $0<T^*=T^*(u_0,q_0,\alpha,f)\leq\infty$ and a unique classical solution $(u(t),q(t))$ to problem \eqref{eq:6b} verifying
$$
u\in L^\infty(0,T^*;H^3(\TT))\cap L^2(0,T^*;\dot{H}^{3+\alpha/2}(\TT)),
$$
$$
q\in L^\infty(0,T^*;H^{2+\alpha/2}(\TT)),
$$
and the following inequality
\begin{multline*}
\sup_{0\leq t \leq T^*}\|u(t)\|_{H^3}^2+2\mu\int_0^{T^*}\|u(s)\|_{\dot{H}^{3+\alpha/2}}^2ds
\\
\leq 2\bigg{(}\|u_0\|_{H^3}^2+\tilde{\gamma}\left(\left\|\pax^3q_0\right\|_{L^2}^2+\left\|q_0\right\|_{L^2}^2\right)\bigg{)},
\end{multline*}
where $\tilde{\gamma}=\tilde{\gamma}_{0}^{2\max\{u_0\}}$ is given by \eqref{gamma}. If $u_0$ verifies the stricter condition
\begin{equation}\label{positiv}
0< u_0,
\end{equation}
then the solution verifies
$$
u\in L^\infty(0,T^*;H^3(\TT))\cap L^2(0,T^*;\dot{H}^{3+\alpha/2}(\TT)),
$$
$$
q\in L^\infty(0,T^*;H^3(\TT)),
$$
and the following inequality
\begin{multline*}
\sup_{0\leq t \leq T^*}\left(\|u(t)\|_{H^3}^2+\gamma\left\|q\right\|_{H^3}^2\right)+2\mu\int_0^{T^*}\|u(s)\|_{\dot{H}^{3+\alpha/2}}^2ds
\\
\leq 2\bigg{(}\|u_0\|_{H^3}^2+\tilde{\gamma}\left(\left\|\pax^3q_0\right\|_{L^2}^2+\left\|q_0\right\|_{L^2}^2\right)\bigg{)},
\end{multline*}
where $\gamma=\gamma_{\min\{u_0\}/2}^{2\max\{u_0\}}$ and $\tilde{\gamma}=\tilde{\gamma}_{\min\{u_0\}/2}^{2\max\{u_0\}}$ are given by \eqref{gamma}. 
\end{teo} 

Here, the sign of $u_0$ plays the role of a stability condition in the same spirit as in Coutand \& Shkoller \cite{Coutand-Shkoller:well-posedness-free-surface-incompressible} and Cheng, Granero-Belinch\'on \& Shkoller \cite{CGS}. In that regard, it helps us to avoid derivative loss.

Before proceeding with the global in time results, we collect some global bounds showing the dissipative character of the system. We define

\begin{equation}\label{Theta}
\Theta(s)=\int_1^s\int_1^\xi \frac{f'(\chi)}{\chi}d\chi d\xi.
\end{equation}

Then,
\begin{teo}\label{dissipation} Let $(u_0,q_0)$ be the initial data satisfying the hypothesis in Theorem \ref{local} and consider $0<\alpha\leq2$ and $\mu>0$. Assume that $f$ is an admissible kinetic function in terms of Definition \ref{admissible} satisfying either  
\begin{enumerate}
\item the bound \eqref{lowerf'} or 
\item $f(y)=y^r/r$, for $1<r\leq2$.
\end{enumerate}
Then, the solution $(u(t),q(t))$ verifies 
\begin{multline}\label{Lyapunovbalance}
\|\Theta(u(t))\|_{L^1}+\frac{1}{2}\|q(t)\|_{L^2}^2+\mu\int_0^t\int_\TT(-\Delta)^{\alpha/2}u\Theta'(u)dxds\\
\leq \|\Theta(u_0)\|_{L^1}+\frac{1}{2}\|q_0\|_{L^2}^2.
\end{multline}
Furthermore, there exists $\mathcal{C}_0(\alpha,u_0,q_0)$ such that, 

\begin{itemize}
\item if $f$ verifies the bound \eqref{lowerf'}, the function $u$ gain the following regularity
$$
\int_0^t\|u(s)\|_{\dot{W}^{\alpha/2-\delta,1}}^2ds \leq C_\delta\mathcal{C}_0,\;\forall\,0<\delta<\alpha/2, \,0<\alpha<2,
$$
$$
\int_0^t\|u(s)\|_{\dot{H}^{\alpha/2}}^{2-\frac{2}{1+\alpha}}ds \leq \mathcal{C}_0,\; \,1<\alpha\leq 2,
$$
and
$$
\int_0^t\|u(s)\|_{\dot{W}^{1,1}}^2ds \leq \mathcal{C}_0,\; \text{ if }\alpha=2,
$$ 
\item if $f(y)=y^r/r$, $1<r<2$, the function $u$ gain the following regularity
$$
\int_0^t\|u(s)\|_{\dot{W}^{\alpha/2r-\delta,r}}^{2r}ds\leq C_\delta\mathcal{C}_0,\;\forall\,0<\delta<\alpha/2r, \,0<\alpha<2,
$$
\item if $f(y)=y^2/2$, the function $u$ gain the following regularity
$$
\int_0^t\|u(s)\|_{\dot{H}^{\alpha/2}}^{2}ds\leq \mathcal{C}_0,\;\forall\,0<\alpha<2,
$$
\end{itemize}
\end{teo} 

Now we proceed with the global in time results. Our first global result regards the hyperviscous case $\alpha=2$:

\begin{prop}\label{globalalpha2} Let $(u_0,q_0)\in H^3(\TT)\times H^3(\TT)$ be the initial data and $f$ be an admissible kinetic function in terms of Definition \ref{admissible} satisfying either 
\begin{enumerate}
\item the bounds \eqref{lowerf'} and
$$
\|f''\|_{C^2}\leq C_2,
$$ 
for suitable constant $0<C_2$ or 
\item $f(y)=y^r/r$, for $1<r\leq2$.
\end{enumerate}
Assume that $0\leq u_0$, $\langle q_0\rangle=0$, $\alpha=2$ and $\mu>0$. Then there exist a unique classical solution $(u(t),q(t))$ to problem \eqref{eq:6b} verifying
$$
u\in L^\infty(0,T;H^3(\TT))\cap L^2(0,T;H^{4}(\TT)),
$$
$$
q\in L^\infty(0,T;H^3(\TT)),
$$
for every $0<T<\infty$. 
\end{prop} 

We introduce our definition of global weak solution:
\begin{defi}\label{weaksol} $(u,q)\in L^\infty(0,T;L^1)\times L^\infty(0,T;L^2)$ is a global weak solution to \eqref{eq:6b} if for all $T>0$, $\phi,\psi\in \mathcal{D}([-1,T)\times \TT)$ we have
$$
\int_0^T\int_\TT -\pat \phi u+\mu u(-\Delta)^{\alpha/2}\phi+ uq\pax\phi dxds+\int_\TT u_0\phi(0)dx=0,
$$
$$
\int_0^T\int_\TT -\pat \psi q+f(u)\pax\psi dxds+\int_\TT q_0\psi(0)dx=0.
$$
\end{defi}

Equipped with Theorem \ref{globalalpha2} we can prove the global existence of weak solutions for \eqref{eq:6b} when $f(u)=u^r/r$, $1\leq r\leq 2$:
\begin{teo}\label{globalweak} Let $(u_0,q_0)\in L^2(\TT)\times L^2(\TT)$ be the initial data and $f(y)=y^r/r$, $1\leq r\leq2$ be the kinetic function. Assume that $0\leq u_0$, $\langle q_0\rangle=0$, $2-r<\alpha\leq 2$ and $\mu>0$. Then there exist at least one global weak solution (in the sense of Definition \ref{weaksol}) $(u(t),q(t))$ to problem \eqref{eq:6b} verifying
$$
u\in L^\infty(0,\infty;L^r(\TT)),
$$
$$
q\in L^\infty(0,\infty;L^2(\TT)).
$$
Furthermore, the solution $u$ gains the following regularity
\begin{itemize}
\item for $r=1$,
$$
u\in L^2(0,\infty;W^{\alpha/2-\delta,1}(\TT)),\;0<\delta\ll1,
$$
\item for $1<r<2$,
$$
u\in L^{2r}(0,\infty;W^{\alpha/2r-\delta,r}(\TT)),\;0<\delta\ll1,
$$
\item for $r=2$,
$$
u\in L^2(0,\infty;H^{\alpha/2}(\TT)).
$$
\end{itemize}
\end{teo}

\section{Preliminaries}\label{sec3}
\subsection*{Notation}
Given $f\in L^1(\TT^d)$, we denote
$$
\langle f\rangle=\frac{1}{|\TT^d|}\int_{\TT^d}f(x)dx.
$$
We write $\mathcal{M}_0$ and $c$ for constant that may change from line to line but only depends on $u_0(x),q_0(x)$ and the kinetic function $f(x)$. We write $\mathcal{P}$ for a generic polynomial that may change from line to line and whose coefficients depends only on $u_0(x),q_0(x)$ and $f(x)$. We consider $\mathcal{H}_\epsilon$ the periodic heat kernel at time $t=\epsilon$.

\subsection*{Functional spaces}
We write $H^s(\TT^d)$ for the usual $L^2$-based periodic Sobolev spaces:
$$
H^s(\TT^d)=\left\{u\in L^2(\TT^d) \text{ s.t. }(1+|\xi|^s)\hat{u}\in l^2\right\},
$$
with norm
$$
\|u\|_{H^s}^2=\|u\|_{L^2}^2+\|u\|_{\dot{H}^s}^2, \quad \|u\|_{\dot{H}^s}=\|\Lambda^s u\|_{L^2}.
$$
The fractional $L^p$-based Sobolev-Slobodeckij spaces, $W^{s,p}(\TT^d)$, are defined as
$$
W^{s,p}=\left\{u\in L^p(\TT^d), \pax^{\lfloor s\rfloor} u\in L^p(\TT^d), \frac{|\pax^{\lfloor s\rfloor}u(x)-\pax^{\lfloor s\rfloor}u(y)|}{|x-y|^{\frac{d}{p}+(s-\lfloor s\rfloor)}}\in L^p(\TT^d\times\TT^d)\right\},
$$
with norm
$$
\|u\|_{W^{s,p}}^p=\|u\|_{L^p}^p+\|u\|_{\dot{W}^{s,p}}^p, 
$$
where
$$
\|u\|_{\dot{W}^{s,p}}^p=\|\pax^{\lfloor s\rfloor} u\|^p_{L^p}+\int_{\TT^d}\int_{\TT^d}\frac{|\pax^{\lfloor s\rfloor}u(x)-\pax^{\lfloor s\rfloor}u(y)|^p}{|x-y|^{d+(s-\lfloor s \rfloor)p}}dxdy.
$$

In the case of functions defined on $\RR^d$ we have definitions with straightforward modifications.

%\subsection{Two equivalent formulations}
%In the one-dimensional case, $d=1$, \eqref{eq:6b} reduces to
%\begin{equation}\label{eq:6b2}
%\left\{\begin{aligned}
%\pat u&=-\mu\Lambda^{\alpha} u+\pax (uq),\\
%\pat q&=\pax f(u),
%\end{aligned}\right.\text{ for }x\in\TT,\,t\geq0,
%\end{equation}
%Given $(u,q)$ a solution \eqref{eq:6b2}, we define the new variable
%$$
%\delta=u-\langle u_0\rangle.
%$$
%Then, equation \eqref{eq:6b2} is equivalent to 
%\begin{equation}\label{eq:6b3}
%\left\{\begin{aligned}
%\pat \delta&=-\mu\Lambda^{\alpha} \delta+\pax (\delta q)+\langle u_0\rangle\pax q,\\
%\pat q&=\pax f(\delta+\langle u_0\rangle),
%\end{aligned}\right.\text{ for }x\in\TT,\,t\geq0,
%\end{equation}

\section{Proof of Lemmas \ref{lemaentropy} and \ref{lemaentropy2}}\label{sec4}
\subsection{Proof of Lemma \ref{lemaentropy}}\label{sec4a}
We write the proof in the case $\Omega=\RR^d$, $0<\alpha<2$, being the case $\Omega=\TT^d$ and the case $\alpha=2$ analogous. Changing variables, we have that
\begin{align*}
I&=\int_{\RR^d}\Lambda^\alpha u(x)\Gamma(u(x))dx\\
&=\int_{\RR^d}\text{P.V.}\int_{\RR^d} \frac{u(x)-u(y)}{|x-y|^{d+\alpha}}\Gamma(u(x))dydx\\
&=-\int_{\RR^d}\text{P.V.}\int_{\RR^d} \frac{u(x)-u(y)}{|x-y|^{d+\alpha}}\Gamma(u(y))dydx.
\end{align*}
Therefore, since $\Gamma$ is non-decreasing,
$$
I=\int_{\RR^d}\text{P.V.}\int_{\RR^d} \frac{u(x)-u(y)}{|x-y|^{d+\alpha}}\left(\Gamma(u(x))-\Gamma(u(y))\right)dydx\geq0,
$$
and, using
$$
\Gamma(u(x))-\Gamma(u(y))=\int_0^1\frac{d}{ds}\Gamma(su(x)+(1-s)u(y))ds,
$$
and \eqref{auxphi}, we compute
\begin{align*}
I&=C(\alpha,d)\int_{\RR^d}\text{P.V.}\int_{\RR^d}\int_0^1 \frac{|u(x)-u(y)|^2}{|x-y|^{d+\alpha}}\Gamma'(su(x)+(1-s)u(y))dsdydx\\
&\geq \frac{C(\alpha,d,\Gamma)}{\|u\|_{L^\infty}}\int_{\RR^d}\text{P.V.}\int_{\RR^d}\int_0^1 \frac{|u(x)-u(y)|^2}{|x-y|^{d+\alpha}}dydx\\
&\geq \frac{C(\alpha,d,\Gamma)}{\|u\|_{L^\infty}}\|\Lambda^{\alpha/2} u\|_{L^2}^2.
\end{align*}
Then, we obtain \eqref{aux1}. For the periodic case, after symmetrizing, we have that
\begin{eqnarray*}
I&=&c_{\alpha,d}\sum_{k\in\mathbb{Z}^d}\int_{\TT^d}
\text{P.V.}\int_{\TT^d}\frac{u(x)-u(y)}{|x-y-2\pi k|^{d+\alpha}}\Gamma(u(x))dydx\\
&\geq& c_{\alpha,d}\int_{\TT^d}\text{P.V.}\int_{\TT^d} \frac{u(x)-u(y)}{|x-y|^{d+\alpha}}\left(\Gamma(u(x))-\Gamma(u(y))\right)dydx.
\end{eqnarray*}
We have
\begin{eqnarray*}
\|u\|_{\dot{W}^{\alpha/2-\delta,1}}&=&\int_{\TT^d}\int_{\TT^d}\frac{|u(x)-u(y)|}{|x-y|^{d+\frac{\alpha}{2}-\delta}}dxdy\\
&=&\int_{\TT^d}\int_{\TT^d}\int_0^1 \bigg[\frac{|u(x)-u(y)|}{|x-y|^{d+\frac{\alpha}{2}-\delta}}\frac{|x-y|^{-\frac{d}{2}+\delta}}{|x-y|^{-\frac{d}{2}+\delta}}\\
&&\times \frac{\Gamma'(su(x)+(1-s)u(y))^{1/2}}{\Gamma'(su(x)+(1-s)u(y))^{1/2}}\bigg]dsdxdy\\
&\leq& I_1^{0.5}I_2^{0.5},
\end{eqnarray*}
with
\begin{align*}
I_1&=\int_{\TT^d}\int_{\TT^d}\int_0^1 \frac{(u(x)-u(y))^2}{|x-y|^{d+\alpha}}\Gamma'(su(x)+(1-s)u(y))dsdxdy\\
&=\int_{\TT^d}
\text{P.V.}\int_{\TT^d}\frac{u(x)-u(y)}{|x-y|^{d+\alpha}}(\Gamma(u(x))-\gamma(u(y)))dydx.
\end{align*}
$$
I_2=\int_{\TT^d}\int_{\TT^d}\int_0^1 \frac{1}{\Gamma'(su(x)+(1-s)u(y))|x-y|^{d-2\delta}}dsdxdy.
$$
Using \eqref{auxphi}, this latter integral is similar to the Riesz potential. Due to the positivity of $u$, we have
\begin{eqnarray*}
I_2&\leq& \frac{1}{c}\int_{\TT^d}\int_{\TT^d}\int_0^1 \frac{su(x)+(1-s)u(y)}{|x-y|^{d-2\delta}}dsdxdy\\
&=&\frac{1}{c}\int_{\TT^d}\frac{1}{|y|^{d-2\delta}}dy\|u\|_{L^1}.
\end{eqnarray*}
Consequently, we get \eqref{aux2}.

\subsection{Proof of Lemma \ref{lemaentropy2}}\label{sec4b}
Using Lemma \ref{lemaentropy} and \eqref{eq:10}, we have that
$$
\|u\|_{\dot{H}^{\alpha/2}(\RR)}^{2-\frac{2}{1+\alpha}}\leq C_2(\alpha,\Gamma)\|u\|_{L^1(\RR)}^{1-\frac{2}{1+\alpha}}\int_{\RR}\Lambda^\alpha u(x)\Gamma(u(x))dx. 
$$
For the periodic case we have \eqref{eq:10b}
$$
\|u\|_{\dot{H}^{\alpha/2}(\TT)}^{2}\leq C(\alpha,\Gamma)\|u\|_{L^1(\TT)}^{1-\frac{2}{1+\alpha}}\left(\|u\|_{\dot{H}^{\alpha/2}(\TT)}^{\frac{2}{1+\alpha}}+\|u\|_{L^1}^{\frac{2}{1+\alpha}}\right)\int_{\TT}\Lambda^\alpha u(x)\Gamma(u(x))dx.
$$
To simplify notation we define
$$
I=\int_{\TT}\Lambda^\alpha u(x)\Gamma(u(x))dx,
$$
so
$$
\|u\|_{\dot{H}^{\alpha/2}(\TT)}^{2}\leq C(\alpha,\Gamma)\|u\|_{L^1(\TT)}^{1-\frac{2}{1+\alpha}}\left(\|u\|_{\dot{H}^{\alpha/2}(\TT)}+\|u\|_{L^1(\TT)}\right)^{\frac{2}{1+\alpha}}I.
$$
Using
$$
(\|u\|_{\dot{H}^{\alpha/2}(\TT)}+\|u\|_{L^1(\TT)})^{2}\leq 2(\|u\|_{\dot{H}^{\alpha/2}(\TT)}^2+\|u\|_{L^1(\TT)}^2),
$$
we obtain
$$
Q^{2}\leq C(\alpha,\Gamma)\|u\|_{L^1(\TT)}^{1-\frac{2}{1+\alpha}}Q^{\frac{2}{1+\alpha}}I+2\|u\|_{L^1(\TT)}^2,
$$
where
$$
Q=\|u\|_{\dot{H}^{\alpha/2}(\TT)}+\|u\|_{L^1(\TT)}.
$$
Finally, we estimate
$$
Q^{2}\leq C(\alpha,\Gamma)\|u\|_{L^1(\TT)}^{1-\frac{2}{1+\alpha}}Q^{\frac{2}{1+\alpha}}I+2\|u\|_{L^1(\TT)}^{2-\frac{2}{1+\alpha}}Q^{\frac{2}{1+\alpha}},
$$
so
$$
\|u\|_{\dot{H}^{\alpha/2}(\TT)}^{2-\frac{2}{1+\alpha}}\leq Q^{2-\frac{2}{1+\alpha}}\leq C(\alpha,\Gamma)\|u\|_{L^1(\TT)}^{1-\frac{2}{1+\alpha}}\left(I+\|u\|_{L^1(\TT)}\right).
$$

\section{Proof of Theorem \ref{local}: Local existence of strong solutions}\label{sec5}
As the construction of suitable regularized problems is not an issue, we focus on the energy estimates. 

Recalling \eqref{gamma}, we define the following energy functional
\begin{equation}\label{energy}
E(t)=\max_{0\leq s\leq t}\left\{\|u(s)\|^2_{H^{3}}+\gamma\|q(s)\|_{H^{3}}^2\right\}+\mu\int_0^t\|u(s)\|_{\dot{H}^{3+\alpha/2}}^2ds.
\end{equation}
Our goal is to obtain an inequality of the type
\begin{equation}\label{inequality}
E(t)\leq \mathcal{M}_0+\sqrt{t}\mathcal{Q}(E(t)),
\end{equation}
for certain constant $\mathcal{M}_0=\mathcal{M}_0(u_0,q_0,f)$ and polynomial $\mathcal{Q}$. The coefficients in $\mathcal{Q}$ depends only on $u_0,q_0$ and $f$. An inequality as \eqref{inequality} implies the existence of $T^*=T^*(u_0,q_0,f)$ such that $E(t)\leq 2\mathcal{M}_0$.

Let us assume first that $u_0$ satisfies \eqref{positiv}. Once this case has been established, in the last step we will recover the case where $u_0$ is non-negative.

\textbf{Step 1: Bootstrap assumptions;}
We assume that $(u,q)$ is a solution verifying
\begin{equation}\label{bootstrap1}
E(t)<3\mathcal{M}_0,
\end{equation}
\begin{equation}\label{bootstrap2}
\min_x u(t)>\frac{1}{4}\min_x u_0>0.
\end{equation}
\begin{equation}\label{bootstrap3}
\max_x u(t)<4\max_x u_0.
\end{equation}
In order we conclude the proof, we will need to prove that stricter bounds hold.

\textbf{Step 2: Positivity and mass conservation;} Given a positive initial data, $u_0>0$, we have that $m(t)=\min_{x}u(x,t)$ solves 
$$
\frac{d}{dt}m(t)\geq m(t)\pax q(x_t,t),
$$
where $x_t\in\TT$ is such that 
$$
\min_x u(x,t)=u(x_t,t),
$$
(see Burczak \& Granero-Belinchon \cite{BG} for a proof). Thus,
$$
\min_x u(x,t)\geq \min_x u_0(x)e^{\int_0^t \pax q(x_s,s)ds}>0
$$
and, consequently,
\begin{equation}\label{consermass}
\|u(t)\|_{L^1}=\|u_0\|_{L^1}.
\end{equation} 
With the same approach,
$$
\max_x u(x,t)\leq \max_x u_0(x)e^{\int_0^t \pax q(x_s,s)ds}.
$$
In particular, notice that we can find $T^1=T^1(u_0,q_0)$ such that
\begin{align}\label{minmax}
\min_{x,t} u(x,t)&\geq \min_x u_0(x)e^{-3c\mathcal{M}_0T^1}>\frac{1}{2}\min_x u_0>0,\\
\max_{x,t} u(x,t)&\leq \max_x u_0(x)e^{3c\mathcal{M}_0T^1}<2\max_x u_0,
\end{align}
if
$$
T\leq T^1.
$$
Thus, the second bootstrap assumptions \eqref{bootstrap2} and \eqref{bootstrap3} hold true. As we are interested in local existence, from this point onwards, we are going to restrict ourselves to $t\in [0,T^1]$.

Notice also that
\begin{equation}\label{consermass2}
\langle q(t)\rangle=\langle q_0\rangle.
\end{equation} 

\textbf{Step 3: Estimates for $u$;} After an integration by parts, we have
$$
\frac{d}{dt}\|u\|_{L^2}^2=-2\mu\int_{\TT}|\Lambda^{\alpha/2}u|^2dx-2\int_{\TT}u q \frac{\partial_t q}{f'(u)}dx.
$$
Using
$$
\int_{\TT}\frac{u}{f'(u)} q \partial_t qdx=\frac{1}{2}\frac{d}{dt}\int_{\TT}\frac{u}{f'(u)} q^2dx-\frac{1}{2}\int_{\TT}\pat \left(\frac{u}{f'(u)}\right) q^2dx,
$$
we obtain
$$
\frac{d}{dt}\left(\|u\|_{L^2}^2+\left\|\sqrt{\frac{u}{f'(u)}}q\right\|_{L^2}^2\right)+2\mu\int_{\TT}|\Lambda^{\alpha/2}u|^2dx=\int_{\TT}\pat\left(\frac{u}{f'(u)}\right) q^2dx.
$$
We have
$$
\int_0^t\int_{\TT}\pat u\left(\frac{1}{f'(u)}-\frac{f''(u)}{(f'(u))^2}\right) q^2dxds\leq tc\max_{0\leq s\leq t}\|\pat u(s)\|_{L^2}^2\|q(s)\|_{L^4}^2,
$$
thus, recalling \eqref{gamma},
\begin{align}
\|u(t)\|_{L^2}^2+\gamma\left\|q\right\|_{L^2}^2+2\mu\int_0^t\|u(s)\|_{\dot{H}^{\alpha/2}}^2ds&\leq \|u_0\|_{L^2}^2+\left\|\sqrt{\frac{u_0}{f'(u_0)}}q_0\right\|_{L^2}^2\nonumber\\
&\quad+t\mathcal{P}(E(t)).\label{uL^2}
\end{align}

\textbf{Step 4: Estimates for $\pax^3 u$;} We compute
$$
\frac{1}{2}\frac{d}{dt}\|u\|^2_{\dot{H}^{3}}+\mu\|u\|^2_{\dot{H}^{3+\alpha/2}}=\int_\TT \pax^4(uq)\pax^3u dx.
$$
Integrating by parts, we have that
\begin{align}
I&=\int_\TT \pax^4(uq)\pax^3udx\nonumber\\
&=-\int_\TT \pax^3(uq)\pax^4udx\nonumber\\
&=-\int_\TT \left(\pax^3uq+u\pax^3q+3\pax u \pax^2q+3\pax q\pax^2u\right)\pax^4udx\nonumber\\
&\leq c\|u\|_{\dot{H}^3}\left(\|u\|_{\dot{H}^3}\|\pax q\|_{L^\infty}+\|u\|_{\dot{W}^{2,4}}\|q\|_{\dot{W}^{2,4}}+\|q\|_{\dot{H}^3}\|\pax u\|_{L^\infty}\right)\nonumber\\
&\quad -\int_\TT u\pax^3q\pax^4udx.\label{eq:I}
\end{align}
In the remainder we have to find an energy term. We compute
$$
\pax^4 u=\frac{\pat \pax^3 q-f'''' (\pax u)^4-6 f''' (\pax u)^2 \pax^2 u-f'' [3 (\pax^2u)^2+4 \pax^3u \pax u]}{f'(u)},
$$
so, by Sobolev embedding,
\begin{align*}
J&=-\int_\TT u\pax^3q\pax^4udx\\
&\leq -\int_\TT u\pax^3q\left(\frac{\pax^3\pat q}{f'(u)}dx\right)+c\|q\|_{\dot{H}^3}\|u\|^2_{\dot{H}^3}\left(1+\|u\|_{\dot{H}^3}^2\right)\\
&\leq -\frac{1}{2}\frac{d}{dt}\int_\TT (\pax^3q)^2\frac{u}{f'(u)}dx+c\|q\|_{\dot{H}^3}\left[\|q\|_{\dot{H}^3}\|\pat u\|_{L^\infty}+\|u\|^2_{\dot{H}^3}\left(1+\|u\|_{\dot{H}^3}^2\right)\right].
\end{align*}
Integrating in time and using \eqref{gamma}, we conclude
\begin{align}
\|u(t)\|_{\dot{H}^3}^2+\gamma\left\|q\right\|_{\dot{H}^3}^2+2\mu\int_0^t\|u(s)\|_{\dot{H}^{3+\alpha/2}}^2ds&\leq \|u_0\|_{\dot{H}^3}^2+\left\|\sqrt{\frac{u_0}{f'(u_0)}}\pax^3q_0\right\|_{L^2}^2\nonumber\\
&\quad+t\mathcal{P}(E(t)).\label{u_xxxL^2}
\end{align}

\textbf{Step 5: Uniform time $T^*$;} Collecting \eqref{uL^2} and \eqref{u_xxxL^2}, we obtain
\begin{align}
\|u(t)\|_{H^3}^2+\gamma\left\|q\right\|_{H^3}^2+2\mu\int_0^t\|u(s)\|_{\dot{H}^{3+\alpha/2}}^2ds&\leq \|u_0\|_{H^3}^2+\left\|\sqrt{\frac{u_0}{f'(u_0)}}\pax^3q_0\right\|_{L^2}^2\nonumber\\
&\quad+\left\|\sqrt{\frac{u_0}{f'(u_0)}}q_0\right\|_{L^2}^2+t\mathcal{Q}(E(t)),\label{En}
\end{align}
and, equivalently,
$$
E(t)\leq \mathcal{M}_0+t\mathcal{Q}(E(t)).
$$
This polynomial inequality implies the existence of $0<T^2=T^2(\mathcal{M}_0,\mathcal{Q})$ such that
$$
E(t)\leq 2\mathcal{M}_0,\;\forall t\leq T^2,
$$
(see Coutand \& Shkoller \cite{Coutand-Shkoller:well-posedness-free-surface-incompressible} or Cheng, Granero-Belinchon \& Shkoller \cite{CGS} for the details). We chose 
$$
T^*=\min\{T^1,T^2\},
$$
where $T^1$ was defined in \eqref{minmax}.

\textbf{Step 6: Uniqueness;} The uniqueness follows a standard approach. Assume that there exists two solutions $(u_1,q_1)$ and $(u_2,q_2)$ with finite energy for the same initial data $(u_0,q_0)$. Define $\bar{u}=u_1-u_2$, $\bar{q}=q_1-q_2$ and $\bar{f}=f(u_1)-f(u_2)$.
We have that
$$
\pat\bar{q}-\bar{f}\pax u_2=f'(u_1)\pax \bar{u},
$$
\begin{align*}
\frac{d}{dt}\|\bar{u}\|^2_{L^2}+2\mu\|\bar{u}\|^2_{\dot{H}^{\alpha/2}}&=-2\int_\TT(\bar{u}q_1-u_2\bar{q})\pax\bar{u}dx\\
&=\int_\TT\bar{u}^2\pax q_1dx-2\int_\TT \frac{u_2}{f'(u_1)}\bar{q}\left(\pat\bar{q}-\bar{f}\pax u_2\right)dx.
\end{align*}
We compute
\begin{align*}
\frac{d}{dt}\left(\|\bar{u}\|^2_{L^2}+\gamma\|\bar{q}\|_{L^2}^2\right)+2\mu\|\bar{u}\|^2_{\dot{H}^{\alpha/2}}&\leq \|\bar{u}\|_{L^2}^2\|\pax q_1\|_{L^\infty}+\|\bar{q}\|_{L^2}^2\left\|\pat\left(\frac{u_2}{f'(u_1)}\right)\right\|_{L^\infty}\\
&\quad+\|\bar{q}\|_{L^2}\|\bar{u}\|_{L^2}\left\|\frac{\pax u_2 u_2}{f'(u_1)}\right\|_{L^\infty}.
\end{align*}
Using Gronwall's inequality, we conclude the uniqueness.

\textbf{Step 7: Non-negative $u_0$;} In the previous steps we have proved that if $u_0>0$, then there exists a unique local solution $(u,q)$ such that
$$
u\in L^\infty_t H^{3}_x\cap L^2_tH^{3+\alpha/2}_x,\;q\in L^\infty_t H^3_x,
$$
where the bound $q\in L^\infty_t H^3_x$ depends on $\min\{u_0\}$. To recover the case with non-negative $u_0$, \emph{i.e.} where $u_0$ may vanish in some region, we consider  the new initial data $u_0^\epsilon=\epsilon+u_0$, where $0<\epsilon\ll1$. For this new initial data we can construct a unique local solution following the previous steps 1-6. Then we have an approximate solution verifying 
\begin{align}
\|u^\epsilon(t)\|_{H^3}^2+2\mu\int_0^t\|u^\epsilon(s)\|_{\dot{H}^{3+\alpha/2}}^2ds&\leq 2+2\|u_0\|_{H^3}^2+2\left\|\sqrt{\frac{\epsilon+u_0}{f'(\epsilon+u_0)}}\pax^3q_0\right\|_{L^2}^2\nonumber\\
&\quad+\left\|\sqrt{\frac{\epsilon+u_0}{f'(\epsilon+u_0)}}q_0\right\|_{L^2}^2,\label{En2}
\end{align}
To pass to the limit, we use that $q$ satisfy
$$
q(x,t)=q_0(x)+\int_0^t f'(u(x,s))\pax u(x,s)ds,
$$
so
$$
\max_{0\leq t\leq T^*}\|q(t)\|_{2+\alpha/2}\leq \|q_0\|_{2+\alpha/2}+C_f\sqrt{T^*}\sqrt{\int_0^{T^*} \|u^\epsilon(s)\|_{3+\alpha/2}^2ds}.
$$
Thus, using \eqref{En2} and the properties of $f$ (see Definition \ref{admissible}), we conclude
$$
q\in L^\infty_t H^{2+\alpha/2}_x,
$$
uniformly in $\epsilon$.

\section{Proof of Theorem \ref{dissipation}: Global bounds}\label{sec6}
\textbf{Step 1: Admissible $f$ satisfying $f'\geq C_1$;} Define the functional
$$
\mathcal{F}[u,q]=\int_{\TT}\Theta(u)dx+\frac{1}{2}\int_{\TT}q^2dx,
$$
where $\Theta$ was defined in \eqref{Theta}. Notice that 
$$
\Theta'(s)=\int_1^s\frac{f'(\chi)}{\chi}d\chi,
$$
which implies that $\Theta'(s)\geq0$ if $s\geq1$ and $\Theta'(s)\leq 0$ if $0< s\leq 1$. We also have
$$
\Theta(1)=0,\Theta'(1)=0,
$$
$$
\Theta''(s)=\frac{f'(s)}{s}\geq0,
$$
which means that $\Theta\geq0$. Thus, the functional $\mathcal{F}$ is bounded below:
$$
0\leq \mathcal{F}[u,q].
$$
Then we have that
\begin{align*}
\frac{d}{dt}\mathcal{F}[u,q]&=\int_\TT \pat u\Theta'(u)dx+\int_{\TT}q\pat qdx\\
&=-\int_\TT\Lambda^{\alpha}u\Theta'(u)dx+\int_\TT \left(-u\Theta''(u)+f'(u) \right)q\pax udx,
\end{align*}
so
\begin{equation}\label{qL^2}
\mathcal{F}[u,q]+\int_0^t\int_\TT\Lambda^{\alpha}u\Theta'(u)dx\leq \mathcal{F}[u_0,q_0].
\end{equation}
As a consequence of this dissipation effect, the conservation of mass \eqref{consermass} and Lemmas \ref{lemaentropy} and \ref{lemaentropy2}, we have the global bounds
\begin{align*}
\int_0^t\|u\|_{\dot{W}^{\alpha/2-\delta,1}}^2ds&\leq C(\alpha,\delta)\|u_0\|_{L^1}\int_0^t\int_{\TT}\Lambda^\alpha u\Theta'(u)dxds\\
&\leq C(\alpha,\delta)\|u_0\|_{L^1}\mathcal{F}[u_0,q_0],
\end{align*}
and, in case $\alpha>1$, 
\begin{align*}
\int_0^t\|u\|_{\dot{H}^{\alpha/2}}^{2-\frac{2}{1+\alpha}}ds&\leq C(\alpha)\|u_0\|_{L^1}^{1-\frac{2}{1+\alpha}}\left(\int_0^t\int_{\TT}\Lambda^\alpha u\Theta'(u)dxds+t\|u_0\|_{L^1}\right)\\
&\leq C(\alpha)\|u_0\|_{L^1}^{1-\frac{2}{1+\alpha}}\left(\mathcal{F}[u_0,q_0]+t\|u_0\|_{L^1}\right).
\end{align*}

Notice that if the dimension is higher than 1, $d\geq2$, the dissipative character of the system remains unchanged and the proof for the cases with higher dimensions follow straightforwardly.

\textbf{Step 2: $f(y)=y^r/r$, $1<r\leq 2$;} Notice that for $f(y)=y^r/r$ we can not apply the argument in Step 1. The reason is that the degeneracy of $f'(y)=y^{r-1}$ is an obstacle for \eqref{auxphi}. Consequently, we can not invoke Lemmas \ref{lemaentropy} and \ref{lemaentropy2} as they are stated. Instead, we notice that in this case we have
$$
\Theta(u)=\int_1^u \frac{1}{r-1}\xi^{r-1}-\frac{1}{r-1}d\xi=\frac{u^r}{r(r-1)}-\frac{u}{r-1}-\left(\frac{1}{r(r-1)}-\frac{1}{r-1}\right).
$$
Thus, computing the evolution of $\mathcal{F}$ and using the conservation of mass, we obtain
\begin{equation}\label{L^rbalance}
\frac{\|u(t)\|_{L^r}^r}{r(r-1)}+\frac{\|q(t)\|_{L^2}^2}{2}+\frac{\mu}{r-1}\int_0^t\int_\TT \Lambda^\alpha uu^{r-1} dxds\leq \frac{\|u_0\|_{L^r}^r}{r(r-1)}+\frac{\|q_0\|_{L^2}^2}{2}.
\end{equation}
Now, in the case $1<r<2$, we invoke Lemma \ref{lemaentropy3} (with $s=r-1$) to obtain the lower bound
$$
0\leq \int_\TT \Lambda^\alpha uu^{r-1} dx,
$$
that implies the uniform-in-time estimates
$$
u\in L^\infty_tL^r_x\cap L^{2r}_tW^{\alpha/2r-,r}_x,\;q\in L^\infty_tL^2_x.
$$
In the case $r=2$, \eqref{L^rbalance} reduces to
\begin{equation}\label{L^2balance}
\|u(t)\|_{L^2}^2+\|q(t)\|_{L^2}^2+2\mu\int_0^t\|u(s)\|^2_{\dot{H}^{\alpha/2}}ds\leq \|u_0\|_{L^2}^2+\|q_0\|_{L^2}^2,
\end{equation}
that implies the uniform-in-time estimates
$$
u\in L^\infty_tL^2_x\cap L^{2}_tH^{\alpha/2}_x,\;q\in L^\infty_tL^2_x.
$$

\section{Proof of Proposition \ref{globalalpha2}: Global existence of strong solutions}\label{sec7}
Fix $0<T<\infty$ an arbitrary parameter and choose $\mu=1$ wlog. Due to Theorem \ref{dissipation}, the solution $(u,q)$ verifies
$$
\int_0^T\|u(s)\|_{L^\infty}^2ds\leq \int_0^T\|u(s)\|_{W^{1,1}}^2ds\leq \mathcal{M}_0,
$$
$$
\max_{0\leq t\leq T}\|q(t)\|_{L^2}^2\leq \mathcal{M}_0.
$$
Then we can refine the previous energy estimates and obtain that
$$
\frac{d}{dt}\|u\|_{L^2}^2+\|u\|_{\dot{H}^1}^2\leq \|u\|_{L^\infty}^2\|q\|_{L^2}^2,
$$
so, integrating,
$$
\max_{0\leq t\leq T}\|u(t)\|_{L^2}^2+\int_0^T \|u(s)\|_{\dot{H}^1}^2ds\leq \mathcal{M}_0.
$$
We also have
\begin{align*}
\frac{d}{dt}\|u\|_{\dot{H}^1}^2+2\|u\|_{\dot{H}^2}^2&= \int_\TT (\pax u)^2\pax qdx-2\int_\TT u\pax q\pax^2udx\\
&\leq c\|u\|_{L^\infty}\|q\|_{\dot{H}^1}\|u\|_{\dot{H}^2}.
\end{align*}
where we have used the inequality
\begin{equation}\label{GN1}
\|g\|_{\dot{W}^{1,4}}^2\leq 3\|g\|_{L^\infty}\|g\|_{\dot{H}^2}.
\end{equation}

We have that
\begin{align*}
\frac{d}{dt}\|q\|_{\dot{H}^1}^2&\leq 2\|f'(u)\pax^2u+f''(u)(\pax u)^2\|_{L^2}\|q\|_{\dot{H}^1}\\
&\leq 2\|u\|_{\dot{H}^2}(f'(0)+4C_2\|u\|_{L^\infty})\|q\|_{\dot{H}^1},
\end{align*}
where we have used
$$
f'(u)\leq f'(0)+C_2u
$$

Then,
\begin{align*}
\frac{d}{dt}\left(\|u\|_{\dot{H}^1}^2+\|q\|_{\dot{H}^1}^2\right)+\|u\|_{\dot{H}^2}^2&\leq c(\|u\|_{L^\infty}+1)\|q\|_{\dot{H}^1}\|u\|_{\dot{H}^2}\\
&\leq c(\|u\|_{L^\infty}+1)^2\|q\|_{\dot{H}^1}^2,
\end{align*}
and, using Gronwall's inequality
$$
\max_{0\leq t\leq T}\|u(t)\|_{\dot{H}^1}^2+\|q(t)\|_{\dot{H}^1}^2\leq \mathcal{M}_0e^{\int_0^Tc(\|u(s)\|_{L^\infty}+1)^2ds}\leq \mathcal{M}_0e^{\mathcal{M}_0(T+1)},
$$
$$
\int_0^T\|u(s)\|_{\dot{H}^2}^2ds\leq\mathcal{M}_0e^{\mathcal{M}_0(T+1)}.
$$

In the same way
\begin{align*}
\frac{d}{dt}\|q\|_{\dot{H}^2}^2&\leq c\left(\|f'(u)\pax^3u\|_{L^2}+\|f'''(u)(\pax u)^3\|_{L^2}+\|f''(u)\pax u \pax^2u \|_{L^2}\right)\|q\|_{\dot{H}^2}\\
&\leq c\left((1+\|u\|_{L^\infty})\|\pax^3u\|_{L^2}+\|\pax u\|^3_{L^6}+\|\pax u \pax^2u \|_{L^2}\right)\|q\|_{\dot{H}^2}.
\end{align*}
By using H\"older's inequality and Gagliardo-Nirenberg interpolation inequalities, we have that
$$
\|\pax u\|^6_{L^6}\leq c\|\pax u\|^6_{\dot{H}^{1/3}}\leq c\|\pax u\|^5_{L^2}\|\pax^3 u\|_{L^2},
$$
\begin{align*}
\|\pax u \pax^2u \|^2_{L^2}&\leq \|\pax u\|^2_{L^6}\|\pax^2u \|_{L^3}^2\\
&\leq c\|\pax u\|^{5/3}_{L^2}\|\pax^3 u\|_{L^2}^{2/3}\|\pax^2 u\|^{5/3}_{L^2}\\
&\leq c\|\pax u\|^{2}_{L^2}\|\pax^3 u\|_{L^2}\|\pax^2 u\|_{L^2},
\end{align*}
and
\begin{equation}\label{qh2}
\frac{d}{dt}\|q\|_{\dot{H}^2}^2\leq \|\pax^3u\|^2_{L^2}+c\|u\|_{\dot{H}^1}^{10}+c\left(\|q\|^2_{\dot{H}^2}+\|\pax^2 u\|_{L^2}\right).
\end{equation}
At this point, to obtain the appropriate bound for $\|u\|_{H^2}^2$ is an easy computation:
\begin{multline}\label{uh2}
\frac{d}{dt}\|u\|^2_{\dot{H}^2}+2\|u\|^2_{\dot{H}^3}\leq c\|u\|_{\dot{H}^3}\left(\|u\|_{\dot{H}^2}\|q\|_{L^\infty}\right.\\
\left.+\|q\|_{\dot{H}^2}\|u\|_{L^\infty}+\|u\|_{\dot{W}^{1,4}}\|q\|_{\dot{W}^{1,4}}\right).
\end{multline}
Adding together \eqref{qh2} and \eqref{uh2} and using \eqref{GN1} and Gronwall's inequality,  we obtain the bound
$$
\max_{0\leq t\leq T}\|u(t)\|_{\dot{H}^2}^2+\|q(t)\|_{\dot{H}^2}^2+\int_0^T\|u(s)\|_{\dot{H}^3}^2ds\leq C(T,\mathcal{M}_0),
$$
The $L^\infty_tH_x^3$ estimates can be obtained with the same ideas and the proof follows.

\section{Proof of Theorem \ref{globalweak}: Global existence of weak solutions}\label{sec8}
Fix $0<T<\infty$ and arbitrary parameter. We consider the approximate problems
\begin{equation}\label{eq:6c}
\left\{\begin{aligned}
\pat u^\epsilon&=-\mu\Lambda^{\alpha} u+\pax(u^\epsilon q^\epsilon)+\epsilon\pax^2u^\epsilon,\\
\pat q^\epsilon&=\pax f(u^\epsilon)+\epsilon\pax^2 q^\epsilon,
\end{aligned}\right.
\end{equation}
with initial data
$$
u^\epsilon(0)=\epsilon+\jeps u_0,\,q^\epsilon(0)=\jeps q_0.
$$

With the same ideas as in Theorem \ref{globalalpha2}, we obtain the global existence of the approximate solution $(u^\epsilon,q^\epsilon)$.

\textbf{Step 1: $f(y)=y^2/2$;} Using Theorem \ref{dissipation} we have the following global $\epsilon$-uniform bounds
$$
u^\epsilon\in L^\infty_tL^2_x\cap L^2_tH^{\alpha/2}_x,
$$
$$
\pat u^\epsilon\in L^2_tH^{-2}_x,
$$
$$
q^\epsilon\in L^\infty_tL^2_x.
$$
Recalling that $0<T<\infty$ and using the embedding $L^\infty\hookrightarrow L^2$, we can apply the classical Aubin-Lions Lemma with
$$
X_0=H^{\alpha/2},\;X=L^2,\;X_1=H^{-2},
$$
so
$$
X_0\subset\subset X \hookrightarrow X_1,
$$
and conclude that
$$
Y=\{v\text{ s.t. }v\in L^2(0,T;X_0)\cap \pat v\in L^2(0,T;X_1)\}
$$
is compactly embedded into $L^2(0,T;L^2)$. Thus, we have the following convergences
\begin{equation}\label{convergence}
u^\epsilon \rightarrow u\text{ in } L^2_tL^2_x,\,u^\epsilon \rightharpoonup u\text{ in } L^2_tH^{\alpha/2}_x,\,q^\epsilon\rightharpoonup q\in L^2_tL^2_x.
\end{equation}

\textbf{Step 2: $f(y)=y^r/r$, $1<r<2$;} As before, using Theorem \ref{dissipation} we have the following global $\epsilon$-uniform bounds
$$
u^\epsilon\in L^\infty_tL^r_x\cap L^{2r}_tW^{\alpha/2r-,r}_x,
$$
$$
\pat u^\epsilon\in L^2_tH^{-2}_x,
$$
$$
q^\epsilon\in L^\infty_tL^2_x.
$$
Using Rellich Theorem we have
$$
W^{\alpha/2r-,r}\subset\subset L^2
$$
provided that
$$
\frac{1}{2}>\frac{1}{r}-\frac{\alpha}{2r},
$$
or, equivalently,
$$
\alpha>2-r.
$$
Using Aubin-Lions with
$$
X_0=W^{\alpha/2r-\delta,r} (\text{ for }0<\delta\ll1 \text{ small enough}),\,X=L^2,X_1=H^{-2}
$$
we obtain the convergences 
\begin{equation}\label{convergence2}
u^\epsilon \rightarrow u\text{ in } L^{2r}_tL^2_x,\,q^\epsilon\rightharpoonup q\in L^2_tL^2_x.
\end{equation}

\textbf{Step 3: $f(y)=y$;} Recalling Theorem \ref{dissipation} we have the following global $\epsilon$-uniform bounds
$$
u^\epsilon\in L^\infty_tL^1_x\cap L^{2}_tW^{\alpha/2-,1}_x,
$$
$$
\pat u^\epsilon\in L^2_tH^{-2}_x,
$$
$$
q^\epsilon\in L^\infty_tL^2_x.
$$
Using Rellich Theorem we have
$$
W^{\alpha/2-,1}\subset\subset L^2
$$
provided that
$$
\alpha>1.
$$
Using Aubin-Lions with
$$
X_0=W^{\alpha/2-\delta,1} (\text{ for }0<\delta\ll1 \text{ small enough}),\,X=L^2,X_1=H^{-2}
$$
we obtain the convergences 
\begin{equation}\label{convergence3}
u^\epsilon \rightarrow u\text{ in } L^{2}_tL^2_x,\,q^\epsilon\rightharpoonup q\in L^2_tL^2_x.
\end{equation}

\textbf{Step 4: Passing to the limit;} Equipped with \eqref{convergence},\eqref{convergence2}, \eqref{convergence3} and the properties of the mollifiers, we can pass to the limit in the linear terms. Thus, we only have to pass to the limit in the nonlinear terms:
$$
I^\epsilon_1=\int_0^T\int_\TT u^\epsilon q^\epsilon \pax \phi dxds,
$$
$$
I^\epsilon_2=\int_0^T\int_\TT f(u^\epsilon) \pax \phi dxds.
$$
We compute
$$
I^\epsilon_1-\int_0^T\int_\TT u q^\epsilon \pax \phi dxds\leq C_\phi\int_0^T\|u^\epsilon-u\|_{L^2}\|q^\epsilon\|_{L^2},
$$
so, using the weak convergence $q^\epsilon\rightharpoonup q$ in $L^2_tL^2_x$, we have
$$
I^\epsilon_1\rightarrow \int_0^T\int_\TT uq\pax\phi dxds.
$$
The case where $f(u)=u$ can be handled easily due to its linearity. Thus, let's focus on the case where the kinetic function is given by $f(u)=u^r/r$, $1<r<2$. We compute
\begin{align*}
I^\epsilon_2- \int_0^T\int_\TT f(u) \pax \psi dxds&=\int_0^T\int_\TT \int_0^1(\lambda u^\epsilon+(1-\lambda)u)^{r-1}(u^\epsilon-u) \pax \psi d\lambda dxds\\
&\leq C_\psi\int_0^T\|(u^\epsilon+u)^{r-1}\|_{L^{r/(r-1)}}\|u^\epsilon-u\|_{L^{r}}ds\\
&\leq C_\psi\int_0^T(\|u^\epsilon\|_{L^{r}}^{r-1}+\|u\|_{L^{r}}^{r-1})\|u^\epsilon-u\|_{L^{r}}ds\\
&\leq C_\psi\sqrt{T}\sqrt{\int_0^T\|u^\epsilon-u\|_{L^{2}}^2ds},
\end{align*}
where we have used the $\epsilon-$uniform boundedness of $u^\epsilon$ in $L^\infty_tL^r_x$. In the final case $r=2$, we have that
\begin{align*}
I^\epsilon_2- \int_0^T\int_\TT u^2/2 \pax \psi dxds&=\int_0^T\int_\TT (u^\epsilon+u)(u^\epsilon-u) \pax \psi  dxds\\
&\leq C_\psi\int_0^T\|u^\epsilon+u\|_{L^{2}}\|u^\epsilon-u\|_{L^{2}}ds\\
&\leq C_\psi \sqrt{T} \sqrt{\int_0^T\|u^\epsilon-u\|_{L^{2}}^2ds}.
\end{align*}
Thus,
$$
I^\epsilon_2\rightarrow \int_0^T\int_\TT f(u) \pax \psi dxds.
$$

\appendix
\section{Fractional Laplacian}\label{sec9b}
Recalling our convention for the Fourier transform:
\begin{align*}
\hat{g}(\xi)&=\frac{1}{(2\pi)^{d/2}}\int_{\RR^d} g(x)e^{-ix\cdot\xi}dx,
\end{align*}
we write $\Lambda^\alpha =(-\Delta)^{\frac{\alpha}{2}}$, \emph{i.e.}
\begin{equation}\label{eq:7}
\widehat{\Lambda^{\alpha}u}(\xi)=|\xi|^\alpha \hat{u}(\xi).
\end{equation}

In this section we are going to obtain the formulation of the fractional Laplacian as the following singular integral
\begin{equation}\label{eq:1b.1.5}
\Lambda^{\alpha}u=2C \int_{\RR^d}\frac{u(x)-u(y)}{|x-y|^{d+\alpha}}dy,
\end{equation}
where
$$
C=\left(\int_{\RR^d}\frac{4\sin^2\left(\frac{x_1}{2}\right)}{|x|^{d+\alpha}} dx\right)^{-1}.
$$

In the case of periodic functions, we have the following equivalent representation
\begin{align}\label{eq:9}
\Lambda^\alpha u(x)&=2C\bigg{(}\sum_{k\in \ZZ^d, k\neq 0}\int_{\TT^d}\frac{u(x)-u(x-\eta)d\eta}{|\eta+2k\pi|^{d+\alpha}}\nonumber\\
&\quad +\text{P.V.}\int_{\TT^d}\frac{u(x)-u(x-\eta)d\eta}{|\eta|^{d+\alpha}}\bigg{)}.
\end{align}
This result is well-known, however, the method that we are going to use has the advantage of being \emph{luddite} in the sense of not requiring any advanced analysis tools, just basic calculus. The main idea is to use the equivalence of norms between $H^s$ and $W^{s,2}$:
\begin{prop}Fix $0<s<1$ and let $u$ be a function in the Schwartz class. Then the following equality holds
\begin{equation}\label{eq:1b.1.2}
\|u\|_{\dot{H}^s(\RR^d)}^2=C\|u\|_{\dot{W}^{s,2}}^2
\end{equation}
for an explicit constant $C=C(s,d)$. 
\end{prop}
\begin{proof}
We compute
\begin{align*}
\|u\|_{\dot{W}^{s,2}}^2&=\int_{\RR^d}\int_{\RR^d}\frac{|u(x)-u(y)|^2}{|x-y|^{d+2s}}dydx\\
&=\int_{\RR^d}\int_{\RR^d}\frac{|u(x+y)-u(y)|^2}{|x|^{d+2s}}dydx\\
&=\int_{\RR^d}\int_{\RR^d}\frac{|e^{i\xi\cdot x}-1|^2|\hat{u}(\xi)|^2}{|x|^{d+2s}}d\xi dx\\
&=\int_{\RR^d}\left(\frac{1}{|\xi|^{2s}}\int_{\RR^d}\frac{4\sin^2\left(\frac{\xi\cdot x}{2}\right)}{|x|^{d+2s}}dx\right) |\xi|^{2s}|\hat{u}(\xi)|^2d \xi,
\end{align*}
due to properties of the Fourier transform.

Due to Plancherel Theorem, the equality \eqref{eq:1b.1.2} now reduces to whether
$$
I(\xi)=\frac{1}{|\xi|^{2s}}\int_{\RR^d}\frac{4\sin^2\left(\frac{\xi\cdot x}{2}\right)}{|x|^{d+2s}} dx
$$
is constant (and then $I=c^{-1}$) or not. Notice that, by changing variables,
$$
I(\lambda \xi)=\frac{1}{|\xi|^{2s}}\int_{\RR^d}\frac{4\sin^2\left(\frac{\xi\cdot \lambda x}{2}\right)}{|\lambda x|^{d+2s}} dx=I(\xi).
$$
Thus, it is enough to consider $\xi$ such that $|\xi|=1$ and
$$
I(\xi)=\int_{\RR^d}\frac{4\sin^2\left(\frac{\xi\cdot x}{2}\right)}{|x|^{d+2s}} dx.
$$
Then, when $d=1$, it is clear that $I(\xi)=I(1)$. When $d=2$, using polar coordinates, we have that $\xi=(\cos(\omega),\sin(\omega))$ and
$$
I(\omega)=\int_{0}^\infty\int_{-\pi}^\pi\frac{4\sin^2\left(\frac{r
\cos(\omega-\theta)}{2}\right)}{r^{1+2s}} d\theta dr.
$$
Thus, changing variables, we have that
$$
I(\omega)=I(0),\text{ i.e. } I(\xi)=I(e_1).
$$
The case where $d=3$ follows with the same ideas. As a consequence, we have proved that
$$
\|u\|_{\dot{W}^{s,2}}^2=\left(\int_{\RR^d}\frac{4\sin^2\left(\frac{x_1}{2}\right)}{|x|^{d+2s}} dx\right)\int_{\RR^d}|\xi|^{2s}|\hat{u}(\xi)|^2d\xi=\left(\int_{\RR^d}\frac{4\sin^2\left(\frac{x_1}{2}\right)}{|x|^{d+2s}} dx\right)\|u\|_{\dot{H}^{s}}^2.
$$
Equivalently, we have proved 
$$
C\|u\|_{\dot{W}^{s,2}}^2=\|u\|_{\dot{H}^{s}}^2
$$
where
$$
C=\left(\int_{\RR^d}\frac{4\sin^2\left(\frac{x_1}{2}\right)}{|x|^{d+2s}} dx\right)^{-1}.
$$
\end{proof}
The previous computation serves as a bridge between the multiplier definition of $\Lambda^{\alpha}$ on the Fourier space and certain integral expression involving a singular kernel on the physical space. Then we have the following result
\begin{prop}Fix $0<s<1$ and let $u$ be a function in the Schwartz class. Then the following equalities holds
\begin{align}\label{eq:1b.1.3}
\|u\|_{\dot{H}^s(\RR^d)}^2&=\int_{\RR^d}\Lambda^{2s}u u dx,\\
\label{eq:1b.1.4}
\|u\|_{\dot{W}^{s,2}}^2&=2\int_{\RR^d}\int_{\RR^d}\frac{u(x)-u(y)}{|x-y|^{d+2s}}dy\, u(x)dx.
\end{align}
Thus
$$
\int_{\RR^d}\Lambda^{\alpha}u u dx=2C\int_{\RR^d}\int_{\RR^d}\frac{u(x)-u(y)}{|x-y|^{d+\alpha}}dy\, u(x)dx,
$$
where
$$
C=\left(\int_{\RR^d}\frac{4\sin^2\left(\frac{x_1}{2}\right)}{|x|^{d+\alpha}} dx\right)^{-1}.
$$
Furthermore, the operators 
$$
T u=\int_{\RR^d}\frac{u(x)-u(y)}{|x-y|^{d+\alpha}}dy, \Lambda^{\alpha}u
$$
are self-adjoint.
\end{prop}
\begin{proof}
Equation \eqref{eq:1b.1.3} is just an easy application of Plancherel Theorem. In the same way, we can obtain the self-adjointness of the fractional Laplacian. To prove equation \eqref{eq:1b.1.4}, we compute as follows
\begin{align*}
J&=\int_{\RR^d}\int_{\RR^d}\frac{u(x)-u(y)}{|x-y|^{d+2s}} u(x) dydx\\
&=-\int_{\RR^d}\int_{\RR^d}\frac{u(x)-u(y)}{|x-y|^{d+2s}} u(y) dydx\\
&=\frac{1}{2}\int_{\RR^d}\int_{\RR^d}\frac{|u(x)-u(y)|^2}{|x-y|^{d+2s}}dydx\\
&=\frac{1}{2}\|u\|_{\dot{W}^{s,2}}^2.
\end{align*}
To see that $Tu$ is self-adjoint, we perform a change of variables,
\begin{align*}
\int_{\RR^d}Tu\,v\,dx&=\int_{\RR^d}\int_{\RR^d}\frac{u(x)-u(y)}{|x-y|^{d+2s}} v(x) dydx\\
&=-\int_{\RR^d}\int_{\RR^d}\frac{u(x)-u(y)}{|x-y|^{d+2s}} v(y) dydx\\
&=\frac{1}{2}\int_{\RR^d}\int_{\RR^d}\frac{(u(x)-u(y))(v(x)-v(y))}{|x-y|^{d+2s}}dydx\\
&=\int_{\RR^d}Tv\,u\,dx,
\end{align*}
and the result follows.
\end{proof}
Once the previous Proposition has been established, we fix $v$ a function in the Schwartz class and consider
\begin{align*}
K=\int_{\RR^d}\Lambda^{\alpha}(u+v)(u+v) dx-2C \int_{\RR^d}T(u+v) (u+v) dx.
\end{align*}
Due to the previous Proposition, we have that
$$
K=0.
$$
Then, we compute
\begin{align*}
K&=\int_{\RR^d}(\Lambda^{\alpha}u+\Lambda^{\alpha}v)(u+v) dx-2C \int_{\RR^d}(Tu+Tv) (u+v) dx\\
&=\int_{\RR^d}\Lambda^{\alpha}uv+\Lambda^{\alpha}vu dx-2C \int_{\RR^d}Tu v+Tvu dx\\
&=2\int_{\RR^d}\Lambda^{\alpha}uv-4C \int_{\RR^d}Tu vdx.
\end{align*}
In particular, we have proved the equality \eqref{eq:1b.1.5}. To obtain \eqref{eq:9}, we decompose $\RR^d$ and use a change of variables.

\section{Fractional Sobolev inequalities}\label{sec9}
We need an interpolation inequality:
\begin{lem}[\cite{Gsemiconductor}]
Fix $1<\alpha\leq 2$. Then, the following inequalities hold
\begin{equation}\label{eq:10}
\|u\|_{L^\infty(\RR)}\leq C(\alpha)\|u\|_{\dot{H}^{\alpha/2}(\RR)}^{\frac{2}{1+\alpha}}\|u\|_{L^1(\RR)}^{1-\frac{2}{1+\alpha}},
\end{equation}
\begin{equation}\label{eq:10b}
\|u-\langle u\rangle\|_{L^\infty(\TT)}\leq C(\alpha)\|u\|_{\dot{H}^{\alpha/2}(\TT)}^{\frac{2}{1+\alpha}}\|u\|_{L^1(\TT)}^{1-\frac{2}{1+\alpha}}.
\end{equation}
\end{lem}
%\begin{proof}
%We prove the interpolation inequality \eqref{eq:10} in the case $\Omega=\RR$. The periodic case follows similarly. If we denote the Fourier transform of $u$ by $\hat{u}$, due to the inequality
%$$
%\|u\|_{L^\infty}\leq \|\hat{u}\|_{L^1},
%$$
%it is enough if we prove the sharper inequality
%$$
%\|\hat{u}\|_{L^1(\RR)}\leq C(\alpha)\|u\|_{\dot{H}^{\alpha/2}(\RR)}^{\frac{2}{1+\alpha}}\|u\|_{L^1(\RR)}^{1-\frac{2}{1+\alpha}}.
%$$
%Fix $r>0$, then we have
%\begin{align*}
%\|\hat{u}\|_{L^1}&=\int_{B(0,r)}|\hat{u}|d\xi+\int_{B(0,r)^c}|\hat{u}||\xi|^{\alpha/2}|\xi|^{-\alpha/2}d\xi\\
%&\leq 2r\|\hat{u}\|_{L^\infty}+\|u\|_{\dot{H}^{\alpha/2}}\sqrt{2}\sqrt{\int_r^\infty \xi^{-\alpha}d\xi}\\
%&\leq 2r\|u\|_{L^1}+\|u\|_{\dot{H}^{\alpha/2}}\frac{\sqrt{2}}{\sqrt{\alpha-1}}r^{\frac{1-\alpha}{2}}.
%\end{align*}
%Taking the appropriate $r$, we conclude.
%\end{proof}

Finally, we collect another inequality

\begin{lem}[\cite{BG4}]\label{lemaentropy3}
Let $0\leq u\in L^{1+s}(\TT)$, $s\le 1$, be a given function and $0<\alpha<2$, $0<\delta<\alpha/(2+2s)$ two fixed constants. Then, 
$$
0 \le \int_{\TT}\Lambda^\alpha u(x) u^s(x)dx.
$$
Moreover,
$$
\|u\|_{\dot{W}^{\alpha/(2+2s)-\delta,1+s}}^{2+2s}\leq C(\alpha,s,\delta)\|u\|_{L^{1+s}}^{1+s}\int_{\TT}\Lambda^\alpha u(x) u^s(x)dx.
$$
\end{lem}

\subsection*{Acknowledgment}
The author is partially supported by the Labex MILYON and the Grant MTM2014-59488-P from the Ministerio de Econom\'ia y Competitividad (MINECO, Spain).

\bibliographystyle{abbrv}
%\bibliography{bibliografia}

\end{document}